\documentclass[a4paper,11pt]{scrartcl}
\usepackage{amsfonts,amsthm,amssymb}
\usepackage{amsmath, amsthm, amssymb,amsfonts}
\usepackage{mathrsfs}
\usepackage{hyperref}
\usepackage[arrow, matrix, curve]{xy}
\usepackage[T1]{fontenc}
\usepackage[latin1]{inputenc}
\usepackage[a4paper]{geometry}
    \newtheorem{theorem}{Theorem}[section]
    \newtheorem{lemma}[theorem]{Lemma}
 \theoremstyle{plain}
    \newtheorem{definition}[theorem]{Definition}
    
    \newtheorem{corollary}[theorem]{Corollary}
 \theoremstyle{remark}
    \newtheorem{remark}[theorem]{Remark}
 \numberwithin{equation}{section} 
 \newcommand{\kommentar}[1]{}
\newcommand{\dx}{\;\mathrm{d}}
\DeclareMathOperator{\tr}{trace}
\newcommand{\Abstract}[1]{\vspace{6mm}\par\noindent\hspace*{10mm} \parbox{140mm}{\small {\bf Abstract.} #1}\par}
\def\R{\mathbb R}
\def\Z{\mathbb Z}
\def\Q{\mathbb Q}

\def\N{\mathbb N}
\def\C{\mathcal{C}\!\ell}

\def\spin{\mathrm Spin}
\def\sp{\mathfrak{spin}}
\def\I{\mathrm I}

\def\G{\mathcal G}
\def\g{\mathfrak g}
\def\H{\mathcal H}

\def\T{\mathbb T}
\def\t{\mathfrak t}

\begin{document}
\title{Diffusive wavelets on the Spin group}
\author{S. Bernstein, S. Ebert, F. Sommen}
\date{}

\maketitle

\Abstract{
The first part of this article is devoted to a brief review of the results about representation theory of the spin group $\spin(m)$ from the point of view of Clifford analysis.
In the second part we are interested in Clifford-valued functions and wavelets on the sphere. The connection of representations of $\spin(m)$ and the concept of diffusive wavelets leads naturally to investigations of a modified diffusion equation on the sphere, that makes use of the $\Gamma$-operator. We will achieve to obtain Clifford-valued diffusion wavelets with respect to the modified diffusion operator $\Delta+\Gamma-\partial_t$.

Since we are able to characterize all representations of $\spin$ and even to obtain all eigenvectors of the (by representation) regarded Casimir operator in representation spaces, it seems appropriate to look at functions on $\spin(m)$ directly. Concerning this, our aim shall be to formulate eigenfunctions for the Laplace-Beltrami operator on $\spin(m)$ and give the series expansion of the heat kernel on $\spin(m)$ in terms of the eigenfunctions of $\Delta_{\spin}$.

}


\tableofcontents

\section{Motivation}
In recent years a growing literature has been focused on the construction of wavelets on 
the sphere (\cite{Ant/Van}, \cite{Ant/Van2}, \cite{Ant/VanN}) and other homogeneous 
spaces, compact groups and stratified groups (\cite{Geller/Mayeli})
These constructions of wavelets have been motivated by strong interests from applied 
sciences such as geophysics (\cite{Freeden/Windheuser}), astrophysics (\cite{holschneider}) but also crystallography (\cite{bernstein/schaeben}, \cite{MMAS_I}, \cite{BHPS}).

One way to construct wavelets on compact groups and homogeneous spaces are so-called 
diffusive wavelets. The idea of diffusive wavelets consists in separating dilations and 
translations when construction wavelets. The dilations are genreated by an diffusion 
process, i.e. from time-evolution of solutions to the heat equation.  
(\cite{Ebert}, \cite{BE-MMAS}, \cite{ebert/wirth}). 

By the description of the importantnce of spinors and the spin group we follow
\cite{JHla}.
A very interesting compact group is the compact spin group $\spin (m).$ The term spin 
was introduced independently in mathematics and physics. In mathematics spinors and 
the spin group are discovered by Èlie Cartan (\cite{cartan}) during his research into the representation of groups.
The spin representation are particular projective representations of the orthogonal and special orthogonal groups. Elements of a spin representation called spinors.


In representation theory it is well known that
there are some representations of the Lie algebra of the orthogonal group which {\it{cannot}} be formed by the usual tensor constructions. These missing representations are then labeled the spin representations, and their constituents spinors. In this view, a spinor must belong to a representation of the double cover of the rotation group $SO(n).$ 
These double-covers are Lie-groups, called spin groups. But it is also true that every Lie algebra can be represented as a bivector algebra; hence every Lie group can be represented as a spin group \cite{DoHeSoA}.

For the compact group $\spin (m)$ all representations are finite-dimensional and their parametrization has been known for many years. For detailed applications to analysis various explicit realizations of these representations are needed. From an analytic point of view, one of the most convinient ways of realizing those representations of $\spin(m)$ that descends to single-valued representations to $SO(m)$ is on polynomials of matrix argument. Polynomials of $k$ vector variables $x_1,\ldots ,x_k$ where $x_l=\sum\limits_{j=1}^m x_{lj}e_j$ can be regarded as polynomials on $\mathbb{R}^{k\times m}$ by the identification
$$ X = \left(\begin{array}{c} x_1 \\ \vdots \\ x_k \end{array}\right) = \left(\begin{array}{ccc} x_{11} & \cdots & x_{1m} \\ \vdots &  & \vdots \\ x_{k1} & \cdots & x_{km}\end{array}\right) = (x_{lj}). $$

The theory of harmonic functions of a matrix variable was presented in detail in \cite{GM}. They consider simplicial harmonics, i.e. harmonic polynomials of a matrix variable invariant under the action of $SL(r),$ which provides models for irreducible representations of $SO(m)$ with integer weight. Based on that in \cite{VanSomCon} models of half integer weights irreducible representations of $\spin(m)$ inside spaces of monogenic functions of several vector variables. This theory was already developed to some extend in \cite{Co} (in case of several quaternionic variables see e.g. \cite{ABLSS}, \cite{Pa}, \cite{Pe}). To obtain polynomial irreducible representations of $\spin(m)$ in \cite{VanSomCon} they look to spaces of polynomials which are already irreducible with respect to the action of $GL(m).$ To obtain models for all integer (half integer) weight representation harmonicity (monogenicity) conditions are imposed.

The connection of representations of $\spin(m)$ and the concept of diffusive wavelets leads naturally to investigations of a modified diffusion equation on the sphere. We will achieve to obtain Clifford-valued diffusive wavelets with respect to the modified diffusion operator $\Delta + \Gamma -\partial_t$ with the $\Gamma$-operator.

\section{Introduction}

The paper is organized as follows. In section 2, after these remarks, we give a short overview on harmonic analysis and representation theory for compact group and basic properties of Clifford algebras and Clifford analysis. Harmonic analysis and representation theory will by used to construct diffusive wavelets. Clifford algebras are a convient way to describe the spin group. 
In section 3 the spin group $\spin(m)$ and its Lie algebra is studied. An important tool for diffusive wavelets are irreducible representations. To construct all reducible representations of $\spin(m)$ we use the Cartan product. Weights of $\spin(m)$ are usually used to label all irreducible representations of the group. For that we will recall roots and wheights of representations as well as the Cartan product. There are two types of fundamental representations of the spin group in the Clifford algebra $\C _m$, i.e. a minimal set of irreducible representations from which every irreducible representation can be build up, are
$$ h(s)a=sas^{-1} \quad \mbox{and} \quad l(s)a = sa . $$
There are many possibilites to realize representations of $\spin(m)$ in $L^2(\C _m).$ We will take the regular representations 
$$ h_r(s):\,f(a) \mapsto f(sas^{-1}) \quad \mbox{and} \quad l_r(s):\,f(a)\mapsto f(sa) $$
and the tensor product representations $h_r \otimes h$ and $l_r\otimes l$ are given by
$$ H(s):\,f(a)\mapsto sf(s^{-1}as)s^{-1} \quad \mbox{and} \quad L(s):\, f(a)\mapsto sf(s^{-1}as).$$
At the end we will need eigenfunctions. For that we prove that function spaces consisting of functions which depends on simplicial variables are invariant under $H(s)$ and $L(s).$
In the final section 4 we construct wavelets on the spin group.
We start with a general description of diffusive wavelets on compact groups and homogeneous spaces. The representation theory of the compact group $SO(n)$ is well-known and the sphere is a homogenous space of $SO(n).$ Therefore, the sphere is also a homogeneous space of the spin group and we will obtain Clifford-valued wavelets therefrom.

\subsection{Fourier Analysis on compact groups}\label{sec_I}

One of the most important theorems of functional analysis is the spectral theorem for compact self-adjoint operators on a Hilbert space. Which states that if
$A$ is a compact self-adjoint operator on a Hilbert space $V$, then there is an orthonormal basis of $V$ consisting of eigenvectors of $A$ and each eigenvalue is real. The analog of this theorem is the Peter-Weyl theorem. We want to recollect some basic notations and properties. 

Let $\G$ be a compact Lie group. A unitary representation of $\G$ is a continuous group homomorphism $\pi$: $\G\to U(d_{\pi})$ of $\G$ into the goup of unitary matrices of a certain dimension $d_{\pi}$ which will be explained later in the Peter-Weyl theorem. Such a representation is irreducible if $\pi(g)M=M\pi(g)$ for all $g\in\G$ and some $M\in \mathbb{C}^{d_{\pi}\times d_{\pi}}$ implies $M=cI$ is a multiple of the identity. Equivalently, $\mathbb{C}^{d_{\pi}}$ does not have non-trivial $\pi$-invariant subspaces $V\subset \mathbb{C}^{d_{\pi}}$ with $\pi(g)V \subset V$ for all $g\in\G.$ Two representations $\pi_1$ and $\pi_2$ are equivalent, if there exists an invertible matrix $M$ such that $\pi_1(g)M=M\pi_2(g)$ for all $g\in\G$.

Let $\hat{\G}$ denote the set of all equivalence classes of irreducible representations. Then this set parametrerizes an orthogonal decomposition of $L^2(\G).$
\begin{theorem}[Peter-Weyl, \cite{Vilenkin_I}] Let $\G$ be a compact Lie group. Then the following statements are true.
\begin{enumerate}
	\item Denote $\H_{\pi}=\{g\mapsto {\rm trace}(\pi(g)M): M\in \mathbb{C}^{d_{\pi}\times d_{\pi}}\}.$ Then the Hilbert space $L^2(\G)$ decomposes into the orthogonal direct sum
	\begin{eqnarray}
		L^2(\G) = \bigoplus_{\pi\in \hat{\G}} \H_{\pi}
	\end{eqnarray}
	\item For each irreducible representation $\pi\in \hat{\G}$ the orthogonal projection  \\ $L^2(\G)\to \mathcal{H}_{\pi}$ is given by
	\begin{eqnarray}
		f \mapsto d_{\pi} \int_{\G} f(h)\chi_{\pi}(h^{-1}g)\,dh = d_{\pi}\,f*\chi_{\pi},
	\end{eqnarray}
	in terms of the character $\chi_{\pi}(g)={\rm trace}(\pi(g))$ of the representation and $dh$ is the normalized Haar measure.
\end{enumerate}
\end{theorem}
We will denote the matrix $M$ in the equation $f*\chi_{\pi}={\rm trace}(\pi(g)M)$ as Fourier coefficient $\hat{f}(\pi)$ of $f$ at the irreducible representation $\pi$. The Fourier coefficient can be calculated as
$$ \hat{f}(\pi) = \int_{\G} f(g)\pi^*(g)\,dg .$$
The inversion formula (the Fourier expansion) is then given by
$$ f(g) = \sum_{\pi\in\hat{G}} d_{\pi}\,{\rm trace}(\pi(g)\hat{f}(\pi)). $$
If we denote by $||M||^2_{HS}={\rm trace}(M^*M)$ the Frobenius or Hilbert-Schmidt norm of a matrix $M,$ then the following Parseval identity is true.
\begin{corollary}[Parseval identity] Let $f\in L^2(\G).$ Then the matrix-valued Fourier coefficients $\hat{f}\in \mathbb{C}^{d_{\pi}\times d_{\pi}}$ satisfy
\begin{eqnarray}
||f||^2 = \sum_{\pi\in\hat{\G}} d_{\pi}\,||f(\pi)||^2_{HS} . \label{parseval_id}
\end{eqnarray}
\end{corollary}
On the group $\G$ one defines the convolution of two integrable functions $f,\,r\in L^1(\G)$ as
$$ f*r(g) = \int_{\G} f(h)r(h^{-1}g)\,dh . $$
Since $f*r\in L^1(\G),$ the Fourier coefficients are well-defined and they satisfy
\begin{corollary}[Convolution theorem on $\G$] Let $f,\,r\in L^1(\G)$ then $f*r\in L^1(\G)$ and
	$$ \widehat{f*r}(\pi) = \hat{f}(\pi)\hat{r}(\pi). $$
\end{corollary}
The group structure gives rise to left and right translations $T_gf\mapsto f(g^{-1}\cdot)$ and $T^gf\mapsto f(\cdot g)$ of functions on the group. A simple computation shows
$$ \widehat{T_gf}(\pi) = \hat{f}(\pi)\pi^*(g) \quad \mbox{and}\quad \widehat{T^gf}(\pi) = \pi(g)\hat{f}(\pi). $$
They are direct consequences of the definition of the Fourier transform.

The existence of the expressions is ensured by compactness and the properties, implied by compactness of $\G$.

Of particular importance are the character functions
\begin{align}
   \chi_\pi(g) &= \tr(\pi(g)).
\end{align}
They are class type, i.e. they are constant on conjugate classes $\{gag^{-1},\,g\in\G\}$. Since every conjugate class intersects the maximal torus $\T$ of $\G$ at least in one point, the characters are completely determined by its values on $\T$. The correspondence between $\hat\G$ and characters is one to one and the characters are eigenfunctions of the Laplace Beltrami operator.  Details about this can be found for example in \cite{Intro_comp_lie_groups,Vilenkin}.

In this way we can also discuss Fourier analysis on homogeneous spaces \cite{ebert/wirth}.

In order to define the Laplace Beltrami operator in a general way we look at the derivative of representations:
\begin{align}
    \pi_*\,v &= \frac{\dx}{\dx t}\left.\pi(\exp(t\omega))\right|_{t=0}v,\qquad v\in\H_\pi,\omega\in \g,
\end{align}
where $\g$ denotes the Lie algebra of $\G$.

There is a one to one correspondence between representations of simply connected Lie groups and Lie algebras. The differential of the representation of a Lie group gives a representation of its Lie algebra.
\begin{center}
\begin{xy}
  \xymatrix{
      \g \ar[r]^\zeta \ar[d]_{Exp}    &   End(\H) \ar[d]^{Exp}  \\
      \G \ar[r]_\pi             &   GL(\H)
  }
\end{xy}
\end{center}
With $d\pi=\zeta$.

Every Vector field on $\G$ can be regarded as first order differential operator. Higher order differential operators from the algebraic point of view rise from the universal enveloping algebra $U_\g$ of $\g$ \cite{D_Bump}Ch.10, \cite{Intro_comp_lie_groups}Ch. 10.

The characteristic property of the Laplacian is to be a left and right invariant differential operator of second order, the corresponding element in $U_\g$ is the central element of second order, the Casimir element.

\begin{definition}
    Let $\g$ and $\G$ be semi-simple. Let $B$ be the killing form and $\{X_i\}$ a orthogonal\footnote[1]{Orthogonality with respect to $B$.} basis of $\g$. Further let $X^i$ be the corresponding dual basis of the dual space of $\g$.
    Then the \emph{Casimir element} is defined by
    \begin{align*}
        \Omega &=\sum_{i=1}^nX_i\otimes X^i
    \end{align*}
    By Riesz representation theorem $X^i$ can be identified with a basis $X_i$ in $\g$.
    \begin{align*}
        \Omega &=\sum_{i=1}^n X_iB(X_i,\cdot)  \in U_{\g}
    \end{align*}
    is in the center of $U_{\g}$ and independent of the choice of $X_i$.
\end{definition}
For a representation $\zeta$ of $\g$, the Casimir operator can be mapped into the representation Hilbert space of $\zeta$, by
\begin{align*}
    \Delta_\G&=\zeta(\Omega) =\sum_{i=1}^n\zeta(X_i)\zeta(X^i).
\end{align*}
Since the differential of a representation of $\G$, gives a representation of $\g$, we can obtain the operator $\dx \pi(\Omega)$ also for representations $\pi$ of $\G$. A natural representation $\zeta$ in the vector space\footnote[2]{In the case of a $\G$ is compact $C^\infty\subset L^2(\G)$ is dense.} $C^{\infty}$ is given by:
\begin{align}\label{rep_casimir}
    X_i & \mapsto \frac{\partial}{\partial x_i},
\end{align}
where $\frac{\partial}{\partial x_i}$ denotes the derivation in $C^\infty$, mapping every $f$ to its Lie derivative. The extension of $\zeta$ to $U_{\g}$ is strait forward, so that
\begin{align*}
    \Delta_\G &= \sum_{i=1}^n\left(\frac{\partial}{\partial x_i}\right)^2
\end{align*}

From the geometric point of view the Laplace Beltrami operator with respect to the Levic-Civita connection $\nabla$ is given as trace of the Hermitian:
\begin{align*}
    \Delta_\G & = -\tr\;\nabla^2.
\end{align*}

Later we will map the Casimir element by other representations than in \eqref{rep_casimir}, in order to obtain the corresponding operator in $\H_\pi$.
\begin{align}
    \pi_*(\Omega) &= \sum_{i=1}^{d_\G}\pi_*(e_i)^2,\label{Casimir_map}
\end{align}
where $\{e_i\}$ form a basis in $\g$, $d_\G$ denotes the dimension of $\g$.

\subsection{Clifford algebra settings}

Clifford algebras arise in many fields. As algebra of operators they play an enormous role in Physics. A realization of it as linear operators on Grassmanian algebra can be found in \cite{GM}, here the realization of the spinor space comes out as the Grassmanian itself.

A comprehensive set of results for Clifford analysis is given by \cite{DSS}. There the realization of the clifford algebra is given for instance as full matrix algebra of appropriate dimension. Since the spinor spaces are minimal left ideals of the algebra, they can be given very convenient in this realization of the Clifford algebra.

To every vectorspace one can associate a corresponding Clifford algebra. Here it is sufficient to define the basic properties of the Clifford algebra as starting point.

Let $\{e_i,\,i=1,...,m\}$ be a basis of $\mathbb{C}^m$; the corresponding Clifford algebra $\C_m$ is determined by the anticommutative relation $-2\delta_{ij} = e_ie_j+e_je_i$\footnote[1]{$\delta_{ij}$ denotes the usual Kroneka symbol}.
\begin{align*}
    \C_m &= \left\{\sum_{A\subset\{1,...,m\}} a_A e_A,\,a_A\in\mathbb{C}\right\},
\end{align*}
where the set $A=\{\alpha_1,...,\alpha_k\}$ shall be sorted, i.e. $\alpha_1<...<\alpha_k$, $k\leq 2^m+1$ and $e_A=e_{\alpha_1}...e_{\alpha_k}$. The dimension of $\C_m$ is $2^m$. The scalars are contained in $\C_m$ as $0$-vectors, hence the unit of $\C_m$ is $1$.

We will make use of the main anti-involution:
\begin{align*}
    \overline a &= \sum_{A\subset\{1,...,m\}} \overline{a_A} \overline{e_A},\qquad \overline{e_ie_j}=\overline{e_j}\overline{e_i},\qquad\overline{e_i} =-e_i.
\end{align*}
The subspace of $\C_m$ of $k$-vectors is given by span$\{e_A,\,|A|=k\}$\footnote[2]{$|A|$ denotes the cardinality of $A$}. The $k$-vectorpart of an $a\in\C_m$ is given by $[a]_k=\sum_{|A|=k}a_Ae_A$ with $|A|=k$. The subspace of $k$-vectors in $\C_m$ is denoted by $\C_{m,k}$.

Also of importance is the Clifford inner product
\begin{align}
    \langle a,b\rangle_{\C_m} &= [\overline a b]_0 = \sum_{|A|=0}^m (-1)^{|A|}\overline{a_A}b_A.\label{bilinearform_C_m}
\end{align}
This makes $\C_m$ to be a Hilbert space with orthonormal bases $\{e_A,\,A\subset\{1,...,m\}\}$. The outer product of in $\C_m$ is defined by
\begin{align*}
    a\wedge b &= \frac12(ab-ba).
\end{align*}

\section{Spin group}

There are several important subgroups in $\C_m$, the Clifford group is defined as set of invertible elements. The pin group is given as the set of products of unit vectors. A vector $a$ is an unit vector, if it is an element of the sphere $S^m$ in the sense that $\mathbb{C}^m$ can be naturally embedded in the $1$-vector space of $\C_m$, i.e. $\sum_{|A|=1}|a_A|^2=1$ and $a_A=0$ for $|A|\neq 1$.

The spin group, in which we are interested is a subgroup of the pin group and is defined as the set of even products of unit vectors
\begin{align*}
    \spin(m) &= \{\prod\limits_{j=1}^{2k}s_j,\,s_j\in S^m\}.
\end{align*}
The group multiplication is the usual Clifford multiplication.


\subsection{Lie algebra of $\spin(m)$}

The Lie algebra $\sp(m)$ of $\spin(m)$ is the space of bi-vectors in $\C_m$: $\sp(m)=\C_{m,2}$. This can be seen as follows:
Since we are in the comfortable situation to expand the exponential mapping $\exp:\sp(m)\to\spin(m)$ in a series, for $X_{ij}=e_{ij}\in\C_{m,2}$ we find:
\begin{align}
    \exp(tX_{jk}) &= \sum_{l=1}^\infty \frac{1}{l!}(\frac12 e_{jk})^l=  e_{jk}\sum_{l=1}^\infty \frac{1}{(2l-1)!}
    t^{2l-1}+\sum_{l=1}^\infty \frac{1}{(2l)!}t^{2l}\nonumber\\
    &= \cos(t)+e_{jk}\sin(t)=e_j(e_k\sin(t)-e_{j}\cos(t))\label{Liealg_rot},
\end{align}
obviously $e_j, (e_k\sin(t)-e_{j}\cos(t))\in S^m$, hence the exponential of an element from $\C_{m,2}$ gives always an element, that can be written es a sum of an even number of unit vectors.

Since $\spin(m)$ is a double covering of $SO(m)$ we have dim\,$\spin(m)=$dim\,$SO(m)=\frac12 n(n+1)$, but this is also the dimension of $\C_{m,2}$ which hence is the complete Lie algebra of $\spin(m)$.

In order to follow the general concept of determining all irreducible representations we need to look at the maximal torus of $\spin(m)$. For convenience of reading we recall the general bases of weights of representations.

\subsection{Roots and weights of representations}

In this section we collect the assertions about weights of representations, that are necessary for the construction of the weights of $\spin(m)$, which are usually used to label all irreducible representations of $\spin(m)$. A more comprehensive discussion about the theoretical bases can be found in $\cite{D_Bump}$, $\cite{Intro_comp_lie_groups}$, $\cite{Vilenkin_III}$ and elsewhere.

We already mentioned that a representations $\pi$ is uniquely determined by the values that it character assumes on $\T$. We now restrict $\pi$ itself to $\T$. What we obtain is the a representation  of $\T$ that decompose into one dimensional irreducible components, since $\T$ is commutative.

Since the torus we are speaking about is compact, all irreducible representations $\pi$ are of the form
\begin{align*}
    \pi: &\; \T \to \{e^{ix}|\;x\in \R\}\\
    t & \mapsto e^{i\theta(t)},
\end{align*}
with an homomorphism  $\theta : \T \to \R/(2\pi\Z)$. Hence $\theta$ itself is  a representation of $\T$. Consequently, the derivative $\dx\theta: \t\to\R$ is a representation of $\t$. This defines the weights of $\pi$:
\begin{definition}
    Let $\pi$ be a representation of $\G$ with $\mbox{dim}(\t)=r$. Let $\pi_j(t)=e^{i\theta_j(t)}$, $j=1,...,r$ be the one dimensional representations in which $\pi$ decompose while its restriction to $\T$.

    We denote the restriction of $\pi$ to $\T$ by $\pi_\T$.

    The set of weights of $\pi$ is given by $\{\pm \dx\theta_j\}\subset\t^*$. If $\pi$ is the adjoint representation, the \emph{weights} are called \emph{roots}.
\end{definition}
If one regards $\pi$ as a matrix with respect to a fixed bases of the representation space its restriction to $\T$ contains $2\times 2$ block matrices (up to change of rows and lines), which correspond to a rotation in the respective plane.
\begin{align*}
    \pi|_\T &= \left(\begin{array}{llllllll}
                    \Theta_1 \\
                    & \Theta_2 \\
                    && \ddots \\
                    &&&\Theta_r\\
                    &&&&1\\
                    &&&&&\ddots\\
                    &&&&&&1
            \end{array}\right),\mbox{ with } \Theta_j=\left(\begin{array}{ll} \cos(\theta_j(t)) & \sin(\theta_j(t))\\
    -\sin(\theta_j(t)) & \cos(\theta_j(t)) \end{array}\right).
\end{align*}

Therefrom we see, that the dimension of the maximal torus is always even. From this on one always has to look at even dimensional and odd dimensional groups separately.

Further note, that the eigenvalues of the derivative of $\pi|_{\T}$ for $X\in\t$ are always pure imaginary:
\[
\frac{\dx}{\dx s}\left. e^{i\theta_j(\exp(sX))}\right|_{s=0}=i\dx\theta_j(X)
\]
and multiplying it with imaginary unit $i$ determines the weights of $\pi$.

The so called \emph{integer lattice} $I$ is defined by $I=\{2\pi I=\exp^{-1}(1)\}\subset \t$.
Since $\dx\theta_j$ are homomorphisms and they are uniquely determined by the values on $\t$, which are mapped to $0\mod2\pi$, $\dx\theta_j$ is determined by its values on the Integer lattice. Additionally under the mapping $\dx\theta_j$, every element from $I$ will by mapped to some integer.

Roughly speaking, the specific form of the weights of the representation $\pi$ corresponds to the density of $\dx\theta_j(2\I)$ in $\Z$. This is meant like follows: let $t_j\in\t$ be so that for every $s\in\R$ (or $\C$) $\dx\theta_j(st_j)$ is zero for all $j$ but exactly one $j\in\{1,...,r\}$. This gives us a direction on $\T$ which we associate to $\theta_j$ and we denote it by $t_j\in\t$. Since there is a smallest $s_j\in \R$ (in $\C$ one with smallest absolute value) so that $\exp(s_jt_j)=1$ and hence
\begin{align}
    \dx\theta_j(s_jt_j)=m_j\in\Z.\label{weightlattice}
\end{align}
Any integer multiple of $(s_jt_j)$ will be mapped to the corresponding integer multiple of $m_j$ in $\Z$. In that sense we mean the density of $\dx\theta_j(I)$ in $\Z$. The correspondence between $\dx\theta_j$ and $m_j$ is one to one, so we will also $m_j$ call weight of $\pi$.

Let $t_1,...,t_r\in\t$ be a normalized (with respect to the killing form) basis of $\t$ and  $m_1,...,m_r$ be the weights of $\pi$, than the mapping
\begin{align}
    \beta:\T & \to \R^n/(2\pi m_1\Z\times...\times2\pi m_r\Z)\\
    \exp\left(\sum_{k=1}^r a_k t_k\right)=\Pi_{k=1}^r \exp(a_k t_k) & \mapsto (a_1,...,a_r)/(2\pi m_1\Z\times...\times2\pi m_r\Z)\label{weights_conn},
\end{align}
gives a embedding of $\T$ in $\R^n$.

In fact for every $(m_1,...,m_r)\in\Z^r_+$ there is a representation $\pi$ with weights $m_1,...,m_r$ as weights. In this way we have labeled the representations by its weights $m_1,...,m_r\in\Z$ and it is necessary to mention the connection \eqref{weights_conn} between $m_j$ and $\dx\theta_j$. One can also chose the lattice so that $(m_1,...,m_r)\in(l\Z)^r_+$ for any $l\in\Q$ as we will see in the case for $\spin(m)$, where the appropriate choice of $l$ will be $\frac12$.

In $\t$ we obtain a lattice corresponding to weights that is given by $\{\sum_{j=1}^r k_j m_j t_j,\,k_j\in\Z\}$. The symmetry of this lattice is of importance and can be expressed by the Weyl group of the corresponding representation.

If there are at least two points in $\T$, that belong to the same conjugate class, then the information about the representation is the same at all these points. Hence we can factor out these symmetry:
\begin{definition}
    The Weyl group is defined by
    \begin{align*}
        W &= N(T)/T,
    \end{align*}
    where $N(T)$ is the normalizer of $T$ in $\G$, i.e. $gTg^{-1}= T\,\forall g\in G\G$
\end{definition}
$W$ acts on $\T$ by conjugation, and hence on $\t$ by the adjoint representation ad$(w)$ for $w\in W$.

The weights of the adjoint representation are called the roots of the representation. We can look at the hyper planes in $\t$ that are the kernel of the roots $\alpha_i$:
\begin{align*}
    L_{\alpha_i}=\{\alpha_i(t)=0\}.
\end{align*}
The complement of the union of all hyper planes consists of open connected components; the closure of every of this components is called to be a Weyl chamber.

The Weyl group permutes the Weyl chambers transitively and hence also the weights, which we can identify with elements in $\t$ by Riesz theorem and which are symmetric to each other in the above sense.

It can be proven, that the reflections at the plains $L_{\alpha_i}$ generates $W$.

One can distinguish an arbitrary Weyl chamber and call it positive. All weights are positive, that are in the dual of the positive Weyl chamber.

\kommentar{If one identifies $\t$ by $\R^r$ one obtains the connection of Weyl group and Coxeter groups, generated by a root system.}

A weight $\dx\theta$ is a highest weight if it is positive and if $\dx\theta-\dx\lambda$ is not positive for all other weights $\dx\lambda$ of the same representation.

Note that in the construction above \eqref{weightlattice}, where we obtained $\dx\theta_j(rs_jt_j)=r m_j$ the vector $(m_1,...,m_n)$ corresponds to the highest weight of the representation.

There is a famous theorem of Weyl, that the correspondence between irreducible representations and highest weights is one to one \cite{D_Bump}.

\subsection{Weights of $\spin(m)$}

In order to get the weights we look at the torus of $\spin(m)$ and its Lie algebra $\t$. The Lie algebra can be given as a maximal system of $\t=\{Y_i,\,i=1,...,r\}\subset \sp(m)$ with $[Y_i,Y_j]=0$ for all $Y_i,Y_j\in\t$. Such a system is obviously given by
\begin{align*}
 \{Y_j=X_{2j-1,2j}=e_{2j-1}e_{2j},\,j=1,...,\left[\frac{m}{2}\right]\}
\end{align*}
and hence $\T=\left\{\prod\limits_{j=1}^{[\frac{m}{2}]} \exp(t_j Y_j),\,t\in[0,2\pi)\right\}$. Hence the weights can be given as the derivative of
\begin{align}
    \theta_j:\T &\to \R/2\pi\\
            \prod\limits_{j=1}^{[\frac{m}{2}]}\exp(t_j Y_j)&\mapsto m_jt_j 2\pi \mod 2\pi,\label{lattice_condition_I}
\end{align}
where the derivative has to be taken with respect to all $t_j$, so that $(m_1,...,m_{[\frac{m}{2}]})$ stands for the weights.
We have to verify, which $(m_1,...,m_{[\frac{m}{2}]})$ are admissible weights.

From \eqref{Liealg_rot} we see, the natural representation of every element $t=\exp(t_je_{j,n-j+1})\in\T$ of the torus is a rotation in the plain $E_f=\mbox{span}\{e_j,e_{n-j+1}\}\subset\mathbb{C}^{m}$ by the angle $m_jt_j2\pi$.

Hence, for any representation $\pi$ of $\spin(m)$ we obtain its restriction to $\T$ as the direct sum of rotations
\begin{align}
    \pi(\prod\limits_{j=1}^{\left[\frac{m}{2}\right]}\exp(t_jY_j))
=
\pi(t_1,...,t_{\left[\frac{m}{2}\right]})v=e^{i(m_1t_1+...+m_{[\frac{m}{2}]} t_{[\frac{m}{2}]})},\label{eigenvect}
\end{align}
for some $(m_1,...,m_{[\frac{m}{2}]})$.

Since the weights corresponds to the dual of the integer lattice in $\T$ we pick out those
$(m_1,...,m_{\left[\frac{m}{2}\right]})$, that $(t_1,...,t_{\left[\frac{m}{2}\right]})\in \mbox{ker}(\exp)\Rightarrow (m_1t_1,...,m_{\left[\frac{m}{2}\right]}t_{\left[\frac{m}{2}\right]})\in \mbox{ker}(\exp)$.

For eigenvalues of rotations the rotation must be by an angle of $0$ or $\pi$.

From \eqref{eigenvect} we see, that for the integer lattice $(m_1t_1,...,m_{\left[\frac{m}{2}\right]}t_{\left[\frac{m}{2}\right]})\in \mbox{ker}(\exp)$ it is $e^{i(m_1t_1+...+m_{[\frac{m}{2}]} t_{[\frac{m}{2}]})}=1$:
\begin{align}
    m_jt_j=0 \mbox{ or } m_jt_j=\pi\mbox{ and }m_1t_1+...+m_{\left[\frac{m}{2}\right]}t_{\left[\frac{m}{2}\right]}=0 \mod 2\pi.\label{lattice_condition}
\end{align}
Consequently, $m_j$ has always to be an integer.

If $t_j=0\mod 2\pi$  one can always remove this component form $\prod\limits_{j=1}^{\left[\frac{m}{2}\right]}\exp(t_jY_j)$, i.e. setting $t_j=0$,  without loosing the property of being element of the integer lattice. If $t_j=\pi\mod 2\pi$ one has to remove additionally another component with the same property in order to stay in the integer lattice.

Hence for any choice of $\varepsilon_j=1$ or $0$ ($j=1,...,\left[\frac{m}{2}\right]$) and
$\varepsilon_1+...+\varepsilon_{\left[\frac{m}{2}\right]}$ is an even integer, $(t_jm_j=\pi\varepsilon_j)$ satisfies \eqref{lattice_condition}.

We assume now, that $(t_1,...,t_{\left[\frac{m}{2}\right]})$ belongs to the integer lattice. Than also $(m_1t_1,...,m_{\left[\frac{m}{2}\right]}t_{\left[\frac{m}{2}\right]})$ shall belong to this lattice. But since for $(\varepsilon_1t_1,...,\varepsilon_{\left[\frac{m}{2}\right]}t_{\left[\frac{m}{2}\right]})$ belongs to it, also $(m_1\varepsilon_1t_1,...,m_{\left[\frac{m}{2}\right]}\varepsilon_{\left[\frac{m}{2}\right]}t_{\left[\frac{m}{2}\right]})$ does, we have that either all $m_j$ are even, or all $m_j$ are odd.

This can also be seen in the easy counter example, where we assume $t_l,t_k=\pi$ and $\varepsilon_j=0$ except $j=k$ and $j=l$. Than $(\varepsilon_1t_1,...,\varepsilon_{\left[\frac{m}{2}\right]}t_{\left[\frac{m}{2}\right]})$ belongs to the integer lattice but $(m_1\varepsilon_1t_1,...,m_{\left[\frac{m}{2}\right]}\varepsilon_{\left[\frac{m}{2}\right]}t_{\left[\frac{m}{2}\right]})$ does only for $m_l$ and $m_k$ both even or both odd.

A discussion about the admissible weights can also be found in \cite{GM}, where here the connection between $m_j$ and $\dx\theta_j$ is another one than we have given by \eqref{weights_conn}, so that the corresponding weights are from $(\frac12\Z)^{\left[\frac{m}{2}\right]}$.

We have to look now at the action of the Weyl group to select the highest weight for every representation.

The Weyl group acts on $\T$ and hence on $\t$ and $\t^*$. Its action on the weights is closed\footnote[1]{The Weyl group maps weight to weights} and corresponds to a permutation of the $m_j$; also change of the sign of $m_j$ is possible. In the case where $m$ is odd, an arbitrary number of sign changes is allowed; while in the case of an even $m$, only an even number of sign changes is possible.

The positive Weyl chamber shall be that one, where
\begin{align*}
    m_1\geq...\geq m_{\left[\frac{m}{2}\right]-1}\geq |m_{\left[\frac{m}{2}\right]}|
\end{align*}
In the case of an odd $m$, all $m_j$ of positive weights are positive. When $m$ is even, $m_{\left[\frac{m}{2}\right]}$ can be negative.

We can also compare weights of different representations by the so called lexicographically order, i.e. $(m_1,...,m_k)<(l_1,...,l_k)$, if the difference $l_j-m_j$ in the first component where the weights are different is positive.

\subsection{Cartan product}

For the construction of all irreducible representations of $\spin(m)$ we make use of the so-called Cartan product.

The Cartan product is a procedure to build up an irreducible representation from two known irreducible representations. Let $\pi_1$ and $\pi_2$ be irreducible representations in $\H_1$ and $\H_2$ respectively, let $(m_1,...,m_k)$ and $(l_1,...,l_k)$ be the highest weights of $\pi_1$ and $\pi_2$. The canonically given representation $\pi_1\otimes\pi_2$ in $\H_1\otimes\H_2$ is highly reducible. The irreducible component of the maximal weight\footnote[1]{with respect to the lexicographically order} occurring in $\pi_1\otimes\pi_2$ has the highest weights $(l_1+m_1,...,l_k+m_k$).

A minimal set of irreducible representations from which we can build up every irreducible representation is called fundamental.

\subsection{Representations of $\spin(m)$}

From the previous we already know, that a fundamental system of irreducible representations of $\spin(m)$ in the case of an odd $m$ is contained in the set of representations with weights of the form $(1,0,...,0),...,(1,....,1)$ and $(\frac12,0,...,0),...,(\frac12,...,\frac12)$ and in the case of $m$ even in the set of representations of weights $(1,0,...,0),...,(1,....,1)$, $(\frac12,0,...,0),...,(\frac12,...,\frac12)$ and $(1,...,1,-1)$, $(\frac12,...,\frac12,-\frac12)$.

For convenience we consider the system of the above form in stead of the (minimal) fundamental system, for which we would not need to call in $(1,...,1)$ or $(1,...,\pm 1)$.

From this starting point the corresponding irreducible representations are worked out in \cite{VanSomCon} as representations in Clifford valued function spaces of spherical monogenics and harmonic functions.

In section \ref{sec_I} we sketched the usual way of harmonic analysis on compact Lie groups. This we will discuss later for functions on $\spin(m)$.

Here we use another way of thinking: we directly use representations of the group (here $\spin(m)$) and its representations to investigate operators in the representation Hilbert space as derivative of representations. Since we are using function spaces as representation Hilbert spaces the operators are derivative operators acting on functions. As the outcome of this we can formulate Clifford valued diffusive wavelets corresponding to a modified diffusion equation, where the corresponding operator is a realization of the Casimir element, just as in the classical case.

There are two types of fundamental representations of the spin group in the Clifford algebra $\C_m$
\begin{align}
    h(s) a &= sas^{-1}\label{fundamental_rep_h}\\
    l(s) a &= s a.\label{fundamental_rep_l}
\end{align}
The invariant subspaces, where $h$ is irreducible are the $k$-vector spaces, these of $l$ are the so called spinor spaces. Obviously they are minimal left ideals in $\C_m$. Spinor spaces can be determined explicit by primitive Idempotents (\cite{DSS},\cite{VanSomCon}). This goes as follow: Set
\begin{align*}
    I_j &=\frac12 (1+i e_je_{j+m}),
\end{align*}
then easily one sees the idempotence $I_j^2=\frac14 (1+2i e_je_{j+m}+(i e_je_{j+m})^2)=\frac14 (1+2i e_je_{j+m}-e_je_{j+m}e_je_{j+m})=I_j$. Furthermore $e_j I_j=\frac12(-ie_{j+m}+e_j)=-ie_{j+m}I_j$ and similar $e_{j+m}I_j=-ie_jI_j$.
A minimal left ideal is generated by $I=I_1...I_m$, namely $\C_{2m} I$. Clearly $I^2=I$.

We introduce also
\begin{align}
    T_j=\frac12(e_{2j-1}-ie_{2j}),\label{idemp}
\end{align}
note that $I_j=T_j\overline T_j$.

There are many possibilities to realize representations of $\spin(m)$ in $L^2(\C_m)$. For instance one can just take the regular representations $h_r$ and $l_r$ of $h$ and $l$ respectively:
\begin{align*}
    h_r(s): f(a) & \mapsto f(sas^{-1})\\
    l_r(s): f(a) & \mapsto f(sa).
\end{align*}
$h_r$ is a representation, which do not distinguish between $h_r(s)$ and $h_r(-s)$ and ends up in a representation, which acts exactly like the usual regular representation of $SO(m)$. Here the double covering nature of $\spin(m)$ with respect to $SO(m)$ reveals.

We are interested in Clifford valued functions.

Applying the regular representations to $L^2(\C_m\to\C_m)$ we can decompose the functions of the representation Hilbert space into a sum of component functions for every $k$-vector component since $\C_m$ forms a vector space:
\begin{align*}
    \sum_{k=0}^{2^{m+1}} f_k(a) e_k.
\end{align*}
This gives nothing new, since now the usual regular representations in $L^2(\C_m)$ can be applied to the component functions $f_k$.

The tensor product representations $h_r\otimes h$ and $h_r\otimes l$ in $L^2(\C_m)\otimes \C_m\simeq L^2(\C_m,\C_m)$ are given by
\begin{align*}
    H(s): f(a) & \mapsto sf(s^{-1}as)s^{-1}\\
    L(s): f(a) & \mapsto sf(s^{-1}as).
\end{align*}
Here the decomposition into component function is not possible since the corresponding components $f_k(a) e_k$ do not longer form an invariant subspace.

\begin{remark}
One important observation is, that the representations are unitary:
\begin{align*}
    \langle H(s) f(a) , H(s) g(a) \rangle_{L^2(\C_m\to\C_m)} &=\int_{\C_m} |s^{-1}f(sas^{-1}) g(sas^{-1})s|^2\dx a\\
    &=\int_{\C_m} |f(sas^{-1}) g(sas^{-1})|^2\dx a.
\end{align*}
Since the action $h_s$ on the argument of the function is exactly the action of rotations the above equation ends up in the regular representation $f(a)\to f(sas^{-1})$ of the rotation group $SO(m)$ which is unitary. Hence $H_s$ is unitary and a similar line shows, that also $L_s$ is.

By unitary of  $H$ and $L$, the invariant subspaces in the representation Hilbert space $L^2(\C_m\to\C_m)$ are orthogonal.

We should assure us, that we are dealing with bounded operators. This follows from compactness of $\spin(m)$:
By smoothness of representations, from compactness follows the finite dimensionality of all irreducible representation spaces and hence the compactness of all derivatives of the representation.
\end{remark}

The most interesting question is now to find the invariant subspaces. This is comprehensively investigated in \cite{VanSomCon}. The desired invariant subspaces are spanned by eigenfunctions of the operators, that one obtain by mapping the Casimir element by the corresponding representation into the representation space.

So we shall look at $H(\C)$ and $L(\C)$. We mentioned already that the space of bivectors $\C_{m,2}$ can be identified as the Lie algebra $\sp(m)$.
We equip it with the natural given killing form $B(\cdot,\cdot)$. A calculation (\cite{preprint_BES}) yields
\begin{align*}
    B(x,y) &=-\frac{1}{4} \sum_{i\neq j} \sum_{k\neq j\atop{k\neq i}}  (  x_{jk}- x_{kj})(y_{kj} - y_{jk}) + ( x_{ki}-  x_{ik}) ( y_{ik} - y_{ki}).
\end{align*}
So that $\|\frac12 e_{ij}\|_B=1$. Consequently we use bases on $\sp(m)$, which is orthonormal with respect to $B$,
\begin{align*}
    \{\frac12 e_{ij},\;1\leq i<j\leq m\}.
\end{align*}
Consequently, as stated in \eqref{Casimir_map} the Casimir element, mapped by
\begin{align*}
    \pi(\Omega) &= \sum_{i,j=1,...,m\atop{i<j}} \pi_\ast\left(\frac12 e_{ij}\right)^2\left(=\frac14 \sum_{i,j=1,...,m\atop{i<j}} \pi_\ast\left( e_{ij}\right)^2 \right).
\end{align*}
In \cite{SO1}, \cite{DSS}, \cite{So_La_Co} and many others we find the calculation of the image obtained of mapping $\Omega$ by $H_\ast$ and $L_\ast$:
\begin{align}
    H_\ast\left(\frac12 e_{ij}\right) &=2(x_j \frac{\dx}{\dx x_i} - x_i \frac{\dx}{\dx x_j}) =: L_{ij}\label{H(e_ij)}
\end{align}
Having in mind, that our representation Hilbert space is a function space, the operator $L_{ij}$ can be interpreted as differential operator along the surface of a sphere. The precise direction is given by the section of the plain, spanned by $x_i$ and $x_j$ and the sphere.
In consequence we have
\begin{align}\label{H(Omega)}
    H_\ast(\Omega) &= \sum_{i,j=1,...,m\atop{i<j}}  H_\ast\left(\frac12 e_{ij}\right)^2 =\sum_{i,j=1,...,m\atop{i<j}} L_{ij}^2,
\end{align}
and further
\begin{align}
    L_\ast\left(\frac12  e_{ij}\right) &= H_\ast\left(\frac12  e_{ij}\right)+\frac12  e_{ij} \bf{1},\label{L(e_ij)}
\end{align}
where $\bf{1}$ denotes the identity operator. Hence
\begin{align}\label{L(Omega)}
    L_\ast(\Omega) &\left(= H_\ast(\Omega) + \sum_{i,j=1,...,m\atop{i<j}} \frac12  e_{ij} H_\ast( \frac12 e_{ij}) + \sum_{i,j=1,...,m\atop{i<j}} \left( \frac12 e_{ij}\right)^2\right)\nonumber\\
    &= H_\ast(\Omega) +\Gamma - \frac14 \left(m+1\atop{2}\right)\bf{1},
\end{align}
with
\begin{align*}
    \Gamma= \sum_{i,j=1,...,m\atop{i<j}} e_{ij}L_{ij}
\end{align*}

We now briefly introduce a special type of functions, of which type the eigenfunctions of $H_*(\Omega)$ and $L_*(\Omega)$ are and which give us the possibility to have a new look at functions on $\spin(m)$..

\subsection{Functions of simplicial variables}\label{sec_simpl_fct}

In this section we show that the function spaces, consisting of functions which depends on simplicial variables, are invariant under $H(s)$ and $L(s)$.

\begin{lemma}
Functions that depend on simplicial variables can be identified with functions on $SO(m)$
\end{lemma}
This can be seen in the following way:

Let $u_1,...,u_m\in \mathbb{C}^m$ be a orthonormal basis in $\mathbb{C}^m$. The corresponding simplicial variable in $\mathbb{C}_m$ is given by
\begin{align*}
    a(u_1,...,u_m) &= u_1+ u_1 \wedge u_2 + u_1 \wedge u_2 \wedge u_3 + ... + u_1\wedge ... \wedge u_m.
\end{align*}
One the other hand one can match a unique righthanded, orthonormal bases to a simplicial variable taking $u_1$ normalized vector. In a second step one tanks a linearly independent vector from the plain that is represented by $u_1\wedge u_2$ and that is spanned by $u_1$ and $u_2$. Now one applies the Gram Schmidt procedure to obtain a righthanded orthonormal bases after $m$ steps.

In what follows we restrict the function of simplicial type to $a(u_1,...,u_m)$, while $u_1,...,u_m$ are assumed to be unit vectors. This gives a one to one correspondence to functions on $SO(m)$.

Further by definition of the outer Product $x\wedge y = \frac12(xy-yx)$ we have
\begin{align*}
    \overline s a(u_1,...,u_m) s = \overline s u_1 s + \overline s u_1 \wedge u_2 s+ \overline su_1 \wedge u_2 \wedge u_3 s+ ... + \overline s u_1\wedge ... \wedge u_ms =a(\overline  s u_1 s,...,\overline s u_m s).
\end{align*}
Consequently $H(s) f(a)=s f(\overline sas)\overline s$ is a function of a simplicial variable, if and only if $f(a)$ is such a function. Hence
\begin{corollary}
Functions of simplicial type are invariant under $H$.
\end{corollary}
Later we will make use of
\begin{lemma}\label{Lemma2}
    A function on the spin group can be represented as a pair of functions of a simplicial variable.
\end{lemma}
\begin{proof}
A function $f(s)$ on $\spin(m)$ can be decomposed in an odd-- and an even Part: $f(s)=\alpha(s)+\gamma(s)$, with $\alpha(s)=\alpha(-s)$ and $\gamma(s)=-\gamma(-s)$. For the odd part $ \gamma(s)$, there is a even function $\beta(s)$ so that $s\beta(s)=\gamma(s)$. Hence the pair $(\alpha,\beta)$ can be identified with $f$. Since $\spin(m)$ is a double covering of $SO(m)$, even functions on $\spin(m)$ can be identified with functions on $SO(m)$. Further all righthanded, orthonormal bases of $\mathbb{C}^m$ can be obtained by the action of exactly one rotation on one of these bases, this identification gives a faithful and irreducible representation (a identification) of $SO(m)$. We have already discussed, that the set of righthanded, orthonormal bases of $\mathbb{C}^m$ are represented by simplicial variables.
\end{proof}

\subsection{Eigenfunctions of $H_*(\Omega)$ and $L_*(\Omega)$ in $L^2(\C_m\to\C_m)$}

There are many places where the eigenfunctions are discussed comprehensively, here we want to recall only the results of the discussion in order to use them for further constructions in the next chapter, where we are more interested in their restriction to the sphere.

One possibility is to realize the irreducible representations in the space of homogeneous Polynomials of vector variables. This way is chosen in \cite{GrossRichards} where the corresponding weight $(m_j)_{j=1}^{[\frac{m}{2}]}$ of the representation is given by the homogenety in the $j$-th vector variable.

For vector variable functions, a rotation -and hence a $H$- invariant differential operator is the Laplacian. The harmonic polynomials satisfy
\begin{align}
    \begin{array}{ll}
    \Delta_{x_i} P(x_1,...,x_k) &= 0 \qquad\mbox{ for }i=1,...,k\\
    \partial_{x_i}\partial_{x_j} P(x_1,...,x_k) &= 0 \qquad\mbox{ for }i\neq j.
    \end{array}\label{harmonic_system}
\end{align}
A monogenic function is given, if
\begin{align}
    \partial_{x_i} P(x_1,...,x_k) &= 0 \qquad\mbox{ for }i=1,...,k.\label{monogenic_system}
\end{align}
Simplicial functions are special kind of functions of vector variables. Its symmetry can be expressed by the characteristic differential equation is
\begin{align}
    \langle x_i\partial_{x_{i+1}}\rangle P(x_1,...,x_k) &= 0 \qquad\mbox{ for }i=1,...,k-1,\label{simplicial_system}
\end{align}
where the definition
\begin{align*}
    \langle x_i\partial_{x_{i+1}}\rangle :=-[x_i\partial_{x_{i+1}}]_0
\end{align*}
is used.

Consequently, the simplicial harmonic system $\mathcal H$ are the polynomials satisfying \eqref{harmonic_system} and \eqref{simplicial_system}; the simplicial monogenics are polynomials, satisfying \eqref{monogenic_system} and \eqref{simplicial_system}.

It can be proven, that the simplicial harmonics are irreducible subspaces spaces for $H$ and the simplicial harmonics are those of $L$.

This is calculated in \cite{VanSomCon} and the highest weight vectors for the weight $(\underbrace{2,...,2}\limits_{k\mbox{ times}},0,...,0)$ is of the form $$\langle x_1\wedge...\wedge
x_k,T_1\wedge...\wedge T_k\rangle_{\C_m}.$$
And the Tensor products,
which we use to represent higher even integer weight representations
$(2s_1,...,2s_k)$, corresponds to the weight vector $$\langle
x_1T_1\rangle_{\C_m}^{2s_1}\langle x_1\wedge
x_2,T_1\wedge T_2\rangle_{\C_m}^{2s_2}...\langle x_1\wedge...\wedge
x_k,T_1\wedge...\wedge T_k\rangle_{\C_m}^{2s_k}.$$

In the case of an odd $m$, for odd integer weights one just has to multiply the weight vectors above from the right by the primitive idempotents $I_1,...,I_k$ in order to obtain the weight of the even integer weight '"$+\frac12$'" in every component.

For the case of an even $m$ there the concept is nearly the same, except for the weights $(2n_1+1,...,\pm(2n_k+1))$. Where the for the plus sign one has to multiply a $I_m$ from the right and for the minus sign one has to multiply $I'=\overline T_m T_m$ (notation from \eqref{idemp}) in the place of $I_m$.

The eigenvalues of $H_*(\Omega)$ for the simplicial harmonic
\begin{align}\label{simpl_harm}
\langle
x_1T_1\rangle_{\C_m}^{m_1}\langle x_1\wedge
x_2,T_1\wedge T_2\rangle_{\C_m}^{m_2}...\langle x_1\wedge...\wedge
x_k,T_1\wedge...\wedge T_k\rangle_{\C_m}^{m_k}\end{align} of the weight $(m_1,...,m_k)$ is
\begin{align*}
    -\sum_{j=1}^kk_j(m_j+m-2j),
\end{align*}
that of $L_*(\Omega)$ for simplicial monogenic
\begin{align}\label{simpl_mon}
\langle
x_1T_1\rangle_{\C_m}^{m_1}\langle x_1\wedge
x_2,T_1\wedge T_2\rangle_{\C_m}^{m_2}...\langle x_1\wedge...\wedge
x_k,T_1\wedge...\wedge T_k\rangle_{\C_m}^{m_k}I_1...I_k\end{align} is
\begin{align*}
    -\sum_{j=1}^km_j(m_j+m-2j+1)-\frac{m(m-1)}{8}.
\end{align*}

\section{Clifford-valued diffusive wavelets on the sphere $S^m$ and the group $\spin(m)$}

\subsection{Idea of diffusive wavelets on groups and homogeneous spaces}\label{sec:idea_dif_wav}

In this subsection we present the concept of diffusive wavelets on a compact Lie group $\G$ and homogeneous spaces $\mathcal X\simeq \G/ \mathscr H$, for a subgroup $\mathscr H$ of $\G$. Later we will follow this concept for construction of clifford valued wavelets on the sphere.

For a detailed discussion of the points of this section see \cite{ebert/wirth}.

In order to obtain wavelets we ask for a translation operator $T(g):L^2(\G)\to L^2(\G)$ and a dilation operator $D_\varrho$ so that we can find a admissible function $\Psi$ so that the family of dilated and translated versions of $\Psi$, namely $\{T(g)D_\varrho \Psi=T(g)\Psi_\varrho,\}$ forms a frame in $L^2(\G)$. The dilation operator is naturally given as
\begin{align*}
	T(g) f(\xi) = f(g^{-1}\xi),\;g\in\G,
\end{align*}
for functions on $\G$. The dilation operator for diffusive wavelets arises from the following construction.

The starting point for the construction is a diffusive approximate identity, which is a special differentiable family of operators. In our case this will be given as fundamental solution of a differential equation.

The interesting properties of the operator family $\{P_t,\,t>0\}$, which we need for the construction are
\begin{align}
\lim_{t\to 0} P_t &= {\bf 1}\label{app_ID_prop_I}\\
\|P_t\| & <M\mbox{ for fixed }M>0,\mbox{ idependant of }t.\label{app_ID_prop_II}
\end{align}
$\|\cdot\|$ denotes the operator norm.

The family of operators will be given as convolution operators while the operators are identified by the the corresponding kernel functions of the convolution.

The desired properties \eqref{app_ID_prop_I} and \eqref{app_ID_prop_II} can be achieved as requirements to the Fourier coefficients of the kernel functions:

\begin{definition}
Let $\hat\G_+\subset\hat\G$ be co-finite .The $C^1$-subfamily $\{p_t,\,t>0\}$ of $L^1(\G)$ is a diffusive approximate identity, if
\begin{align*}
\|\hat p_t(\pi)\| & \leq M\mbox{ for some constant }M>0,\mbox{ idependant of }t>0\mbox{ and }\pi\in\hat\G\\
\lim_{t\to 0} \hat p_t(\pi)& ={\bf 1}\;\forall \pi\in \hat\G\\
\lim_{t\to \infty} \hat p_t(\pi)& =0\;\forall \pi\in \hat\G_+\\
 -\partial_t\hat{p}_t & \mbox{ is a positive matrix for all }t>0\mbox{ and }\pi\in\hat\G_+.
\end{align*}
\end{definition}
Also in what follows it is convenient to look at the Fourier domain and which action to the Fourier coefficients corresponds to the constructions respectively.

We make use of the involution
\begin{align*}
	\check{f}(g)&:=\overline{f}(g^{-1})& \hat{\check{f}}(\pi) = \hat{f}^*(\pi).
\end{align*}
\begin{definition}
Let $\{p_t,\,t>0\}$ be a diffusive approximate identity. A diffusive wavelet is a subfamily $\{\Psi_\varrho,\,\varrho>0\}$ of $L^2(\G)\cap L^1(\G)$, with
\begin{align}
	p_t(g) &= \int_t^\infty \check\Psi_\varrho \ast\Psi_\varrho(g)\alpha(\varrho)\dx\varrho,\label{admis_cond}
\end{align}
for some weight function $\alpha(\varrho)>0$.
The wavelet transform then is given as bounded operator from $L^2(\G)$ to $L^2(\G\times\R_+)$\footnote[2]{the measure is given by the product of the Haar measure on $\G$ and $\alpha(\varrho)\dx\varrho$}.
\begin{align*}
	\mathcal W f(g,\varrho) &:=(f\ast \check{\Psi}_\varrho)(g)=\langle f,T(g)\Psi_\varrho\rangle_{L^2(\G)}.
\end{align*}
\end{definition}
The reconstruction formula assumes the form
\begin{align*}
	f(g)  &= \int_{t\to 0}^\infty (\mathcal W f(\cdot,\varrho)\ast\Psi_\varrho)(g) \alpha(\varrho)\dx\varrho.
\end{align*}
The reconstruction formula can now be verified as follows:
\begin{align*}
	\int_{t\to 0}^\infty ((f\ast\check\Psi_\varrho)\ast\Psi_\varrho)(g) \alpha(\varrho)\dx\varrho
	&=\left(f\ast \left(\int_{t\to 0}^\infty \check\Psi_\varrho\ast \Psi_\varrho \alpha(\varrho)\dx\varrho\right)\right)(g)\\
	&=\lim_{t\to0}(f\ast p_t)(g)=f(g),
\end{align*}
by admissibility condition \eqref{admis_cond} and by the definition of $p_t$ the necessary change of order of integration is valid.

The same construction can be projected to arbitrary homogeneous spaces $\mathcal X\simeq\G/\mathscr{H}$, for any subgroup $\mathscr H$ of $\G$.

The translation is again given by the canonical defined action of $\G$ on $\mathcal X$, which we denote by $g\cdot x$ ($g\in\G,\;x\in \mathcal X$).

Let $\mathfrak P_\mathscr H(\pi)$ be the projection in $\H_\pi$ onto the (with respect to the action of $\pi$) $\mathscr H $ invariant subspace in $\H_\pi$.
A function on $\G$, that is constant along fibers of the form $g\mathscr H$ can be regarded as function on $\mathcal X$. The projection of such a function is denoted by
\begin{align*}
	\mathbb P_\mathcal X f (x) &= \int_\mathscr f(gh)\dx h,\;x=g\cdot\mathscr H.
\end{align*}

The Fourier coefficients of those functions are of the following form:
\begin{align*}
	\hat f(\pi) & =\mathfrak P_{\mathscr H}(\pi)\hat f(\pi),
\end{align*}
i.e. one can always chose a basis in $\H_\pi$, so that the matrix, corresponding to $\hat f(\pi)$ has only entries in the first $k$ rows, where $k$ is the dimension of the $\mathscr H$ invariant subspace in $\H_\pi$. (see \cite{ebert/wirth}).

So the construction can be obtained, if one repeats it on the Fourier domain, with respect to Fourier coefficients of the appropriate form and translate it back to the integral expressions on $\G$ and $\mathcal X$. We have to define the following convolution like products
\begin{align*}
     &\mbox{Group convolution }& (f\ast h)(y) & =\int_\mathcal G f(g\cdot x_0) h(g^{-1}\cdot y)\dx g & \widehat{f \ast h} &= \hat h \hat f\\
     &\mbox{Dot product }& (f \bullet h) (g) &= \int_\mathcal X f(x) \overline{h(g^{-1}\cdot x)} \dx x & \widehat{f \bullet h} & = \hat h^* \hat f\\
     &\mbox{Zonal product }& (f \hat\bullet h) (y) &=\int_\mathcal G \overline{f(g\cdot x_0)}h(g \cdot y)\dx g & \widehat{f \hat\bullet h} &= \hat h \hat f^*
\end{align*}

Since the heat kernel on $\mathcal X$ can be obtained by projection of that on $\mathcal G$ we want to reformulate the defining equation on $\mathcal X$.
\begin{align*}
    \mathbb P_\mathcal X p_t^{heat}(g)&=p_{\mathcal X}^{heat}(t,x)=\int_t^\infty \mathbb P_\mathcal X \left(\check\Psi_\rho\ast \Psi_\rho\right) \alpha(\rho)\dx\rho\\
    &=\int_t^\infty\mathbb P_\mathcal X\left( \int_\mathcal G \overline{\Psi_\rho(a^{-1})}\Psi_\rho(a^{-1}g)\right)\dx a \alpha(\rho)\dx\rho\\
    &=\int_t^\infty \mathbb P_\mathcal X(\Psi_\rho)\, \hat\bullet\, \mathbb P_\mathcal X(\Psi_\rho)(x) \alpha(\rho)\dx\rho,\qquad x=g\cdot x_0.
\end{align*}

\begin{definition}[Wavelet transform on $\mathcal X$]
Let $\{\Psi_\rho,\,\rho>0\}$ be a diffusive wavelet on $\mathcal X$ and $f\in L^2(\mathcal X)$,
\begin{align*}
    WT_\mathcal X(f)(\rho,g) &:= (f\bullet\Psi_\rho)(g) & (=\langle f,\,T_g\Psi_\rho\rangle_{L^2(\mathcal X)}) 
\end{align*}
\end{definition}
The reconstruction now is
    \begin{align*}
        f &= \mathbb P_\mathcal X\int_{\rightarrow 0}^\infty WT_\mathcal X(f)(\rho,\cdot) \ast \tilde \Psi_\rho \alpha(\rho) \dx\rho = \int_{\rightarrow 0}^\infty f\bullet(\Psi_\rho \hat\bullet\Psi_\rho)\alpha(\rho) \dx\rho
    \end{align*}

A special class of functions are zonal functions;
\begin{definition}
 A function on $\mathcal X$ is zonal with respect to $x_0\in X$, if it is invariant under the action of the stabilizer of $x_0$
\end{definition}

These additional symmetry of zonal wavelets can be unrealized, since in that case $\check\Psi_\rho$ is well defined also on $\mathcal X$.

sThe wavelet transform can be written on $\mathcal X$:
\begin{align*}
    \int_\mathcal G \tilde f(a) \check\Psi_\rho(a^{-1}g) \dx a &= \int_\mathcal G \tilde f(a) \overline{\Psi_\rho(g^{-1}a)} \dx a= \int_\mathcal X f(x)\overline{\mathbb P \Psi_\rho(g^{-1} x)}\dx x\\
    &=:WT(f)(\rho,y),\qquad y=g\cdot x_0.
\end{align*}

And the reconstruction succeeds by
\begin{align*}
    f&=\int_0^\infty WT(\rho,\cdot)\ast \Psi_\rho \alpha(\rho)\dx\rho,
\end{align*}
where $\ast$ here denotes the (group) convolution on $\mathcal X$.

The representation theory of $SO(n+1)$ is very well know and we haves seen in the previous section \ref{sec:idea_dif_wav} how we can obtain wavelets on the sphere as a homogeneous space of $SO(n+1)$, of course the sphere is also a homogeneous space of the spin group and we want to obtain wavelets therefrom.

Since the representations $H$ and $L$ act on the argument of the function by a rotation, the invariant functions will be defined on rotation invariant subspaces. We utilize this fact to consider only functions on the sphere $S^m=\{u\in\C_{m+1},\,\sum_{A}u_A e_A=\sum_{j=1}^{m+1} u_je_j,\,\langle u,u\rangle_{C_{m+1}}=1\}\subset\C_{m+1}$.

Since functions on the sphere depend only on one vector, one sees no longer their simplicial character. In case of simplicial monogenic functions of degree $k$, after this restriction we end up with the space of spherical monogenics of degree $k$. Following \cite{DSS} this space shall be denoted by $\mathcal M(k,V)$ or $\mathcal M(m,k,V)$ if we wish to emphasize the dimension of the sphere. Values of spherical monogenics are in $V$ which is chosen to be a spinor space or the whole Clifford algebra.

The spherical monogenics decompose further into two disjoined subspaces, namely
\begin{itemize}
\item The so-called inner spherical monogenics, i.e. homogeneous monogenic polynomials of degree $k$ (harmonics of order $k$): $\mathcal M^+(m,k,V)$
\item The so-called outer spherical monogenics, i.e. homogeneous monogenic functions of degree $-(k+m)$ (harmonics of order $k+1$): $\mathcal M^-(m,k,V)$
\end{itemize}
$\mathcal M^+(m,k,V)$ and $\mathcal M^-(m,k,V)$ are eigenspaces of the Gamma operator:
\begin{align}
\begin{array}{ll}
    \Gamma_\xi P_k(\xi) &= (-k)P_k(\xi),\qquad\forall P_k\in \mathcal M^+(m,k,V)\\
    \Gamma_\xi Q_k(\xi) &= (k+m)Q_k(\xi),\qquad\forall Q_k\in \mathcal M^-(m,k,V).
\end{array}\label{monogenics_eigenv}
\end{align}
and of the spherical Laplace-Beltami operator $\Delta_\xi$:
\begin{align}
    \Delta_\xi P_k(\xi) = H_*(\Omega)P_k&= (-k)(k+m-1)P_k(\xi),\qquad\forall P_k\in \mathcal M^+(m,k,V)\\
    \Delta_\xi Q_k(\xi) = H_*(\Omega)Q_k&= -(k+1)(k+m)Q_k(\xi),\qquad\forall Q_k\in \mathcal M^-(m,k,V)\label{harmonics_eigenv}
\end{align}

The theory of these function systems is well described in \cite{DSS} and elsewhere. There one finds the decomposition
\begin{align*}
    L^2(S^m,\C_{m+1})=\bigoplus\limits_{k=0}^\infty (\mathcal M(k,\C_{m+1}))=\bigoplus\limits_{k=0}^\infty(\mathcal M^+(k,\C_{m+1})\oplus\mathcal M^-(k,\C_{m+1}))
\end{align*}
and $P_k$ and $Q_k$ form an orthogonal basis with respect to the $L^2$ scalar product
\begin{align*}
    \langle f,g\rangle_{L^2} &= \int_{S^m} \langle \overline{f(\xi)} g(\xi)\rangle_{\C_{m+1}}\dx\xi.
\end{align*}
The space of harmonic functions clearly contains the monogenic functions. The space of $k$-homogeneous functions $\mathcal H(m,k,\C_{m+1})$ can be decomposed into
\begin{align*}
    \mathcal H(m,k,\C_{m+1})=\mathcal M^+(m,k,\C_{m+1})\oplus \mathcal M^-(m,k-1,\C_{m+1}).
\end{align*}

Consequently, considering \eqref{L(Omega)} and \eqref{monogenics_eigenv} we have
\begin{itemize}
\item The space of spherical monogenics $\mathcal M(k,\C_{m+1})=\mathcal M^+(k,\C_{m+1}\oplus\mathcal M^-(k,\C_{m+1})$ forms the eigenspace of $L_*(\Omega)$ with respect to the eigenvalue $(-k)(k+m)-\left({m+1}\atop 2\right)$, i.e.
\end{itemize}
\begin{align*}
    L_*(\Omega)P_k&= (-k(k+m)-\frac14\left(m+1\atop{2}\right))P_k \\
    L_*(\Omega)Q_k&= (-k(k+m)-\frac14\left(m+1\atop{2}\right))Q_k.
\end{align*}
From \eqref{harmonics_eigenv} and \eqref{H(Omega)} one sees
\begin{itemize}
\item The space of harmonic functions $\mathcal H(k;\C_{m+1}) =\mathcal M^+(k,\C_{m+1})\oplus\mathcal M^-(k-1,\C_{m+1})$ forms the eigenspace of $H_*(\Omega)$ with respect to the eigenvalue $(-k)(k+m)$, i.e.
    \begin{align*}
     H_*(\Omega)P_k&= -k(k+m)P_k(\xi)\\
    H_*(\Omega)Q_{k-1}&= -k(k+m)Q_{k-1}(\xi).
    \end{align*}
\end{itemize}

For concrete calculations one has to construct the functions $P_k$ and $Q_k$. Let \newline$\alpha=(\alpha_1,...,\alpha_{m+1})\in\N^{m+1}$ denote a multi-index, with the usual notations
\begin{align*}
    x^\alpha  &= x_1^{\alpha_1}...x_{m+1}^{\alpha_{m+1}}\quad\mbox{ for }x\in\C^{m+1}\\
    \partial^\alpha =\partial_{x_1}^{\alpha_1}...\partial_{x_{m+1}}^{\alpha_{m+1}}\\
    \alpha! &= \alpha_1!...\alpha_{m+1}!\\
    |\alpha| &= \sum_{j=1}^{m+1}\alpha_j
\end{align*}
Starting from a natural system of polynomials, namely $\{\frac{1}{\alpha}\xi^\alpha\}$, a system of monogenic functions can be given as Cauchy-Kovalevskaya extension of these polynomials
\begin{align*}
V_\alpha(\xi)=CK\left(\frac{1}{\alpha!}\xi^\alpha\right)=\sum_{j=0}^{|\alpha|}\frac{(-1)^j\xi_0^j}{j!}[(\overline{e_0}\partial_\xi)^j\xi^\alpha].
\end{align*}
For details we refer to \cite{DSS}. A basis of $\mathcal M^+(k,\C_m)$ is given by the set:
\begin{align*}
    \{V_\alpha,\,|\alpha|=k\}.
\end{align*}


Defining further
\begin{align*}
    W_\alpha(\xi) &= (-1)^{|\alpha|}\partial^{\alpha}\frac{\overline\xi}{A_m},
\end{align*}
where $A_m$ denotes the area of $S^m$, a basis of $\mathcal M^-(k,\C_{m+1})$ can be given by
\begin{align*}
    \{W_\alpha,\,|\alpha|=k\}.
\end{align*}
Further expansions can be found in \cite{DSS}.

\begin{align*}
C^-_{m+1,k}(\omega,\xi)&= \frac{1}{m-1}[(k+1) C^{(m-1)/2}_{k+1}(t)+(1-m)C_k^{(m+1)/2}(t)((\xi_0\underline\omega-\omega_0\underline\xi)e_0+\underline\omega\wedge\underline\xi)],\\
C^-_{m+1,k}(\omega,\xi)\overline\xi &= C^+_{m+1,k}(\omega,\xi)\overline\omega.
\end{align*}
A left and right monogenic polynomial of degree $k$ in $x\in\C^{m+1}$ (for fixed $y\in\C_0^{m+1}$) is given by $\frac{|x|^k}{|y|^{m+k}}C^+_{m+1,k}(\omega,\xi)\overline\omega$ (it is $x=|x|\xi$ and $y= |y|\omega$ )

With these function systems we are now in the condition to apply our method of constructing diffusive wavelets in the same way we did it for scalar-valued functions on the sphere (\cite{MMAS_I}.

\subsection{Heat kernel of \kommentar{$H_\ast(\Omega)-\partial_t$ and} $L_\ast(\Omega)-\partial_t$}

Since $H_\ast(\Omega)$ is the usual spherical Laplace Beltrami operator, $H_\ast(\Omega)-\partial_t$ is the usual heat operator. Here we use it to give the fundamental solution $P_H$ for Clifford valued functions, which has the series expansion:
\begin{align*}
    P_H(\omega,\xi) &= \sum_{k=0}^\infty \sum_{|\alpha|=k}\sum_{|\beta|=k-1} Exp\left(-k(k+m-1)t\right)(V_\alpha(\xi)+W_\beta(\xi)).
\end{align*}
The series expansion of the fundamental solution of $L_\ast(\Omega)-\partial_t$ can now be given as:
\begin{align*}
    P_L(\xi) &= \sum_{k=0}^\infty \sum_{|\alpha|=k} Exp\left(-\left(k(k+m)-\left({m+1}\atop 2\right)\right)t\right) (V_\alpha(\xi)+W_\alpha(\xi)).
\end{align*}

For our special desire we do not need to formulate here the Fourier theory in full generality, therefor we recommend \cite{ebert/wirth}. As we already mentioned we can expand $f\in L^2(S^m,\C_m)$ in spherical monogenics
\begin{align*}
    f(\xi) &= \sum_{k=0}^\infty \sum_{|\alpha|=k} (\hat f_V(\alpha) V_\alpha(\xi)+\hat f_W(\alpha) W_\alpha(\xi)).
\end{align*}

There are many ways to assume the sphere as a homogeneous space; here  we look at it as $S^m\simeq SO(m+1)/SO(m)$.
Consequently, the convolution of two functions $f,h\in L^2(S^m,\C_m)$ gives a function on $SO(m)$, which is constant over co-sets $gSO(m-1)$ and hence defines a function on $S^m$ \cite{ebert/wirth}.
\begin{align*}
    (f\ast h)(\xi;\omega)=\int_{SO(m+1)} \overline{f(g(\xi))}h(g(\omega))\dx g,
\end{align*}
where $\dx g$ is taken as the Haar measure and $g(\xi)$ stands for element, obtained the rotation $g$ applied to $\xi$.
We shall look at the invariance properties of the convolution product.
There is an $\eta\in SO(m)$ with
\begin{align*}
    (f\ast h)(\xi;\eta(\xi)) &=(f\ast h)(g(\xi);g(\eta(\xi)))=:(f\ast h)(\eta)\;\forall g\in SO(m).
\end{align*}
Since this $\eta$ is not unique but can be chosen as $\eta\zeta$, with $\zeta$ coming from the stabilizer of $\eta$ we find $(f\ast h)(\eta SO(m-1))=(f\ast h)(\eta)$ for a certain subgroup (of the form) $SO(m-1)$ in $SO(m)$, by factoring this subgroup $(f\ast h)$ becomes a function on $S^m$.

On the other hand this function is invariant under the left action of the stabilizer in $SO(m)$ of $\xi$, which also is of the form $SO(m-1)$. functions with this property are called to be zonal.

One can formulate the convolution theorem, which assumes the form
\begin{theorem}
    For $f,h\in L^2(S^m,\C_m)$ it is
    \begin{align*}
        f\ast g= \sum_{k=0}^\infty \sum_{|\alpha|=k} (\hat f_V(\alpha)\hat h_V(\alpha) V_\alpha(\xi)+\hat f_W(\alpha)\hat h_W(\alpha)W_\alpha(\xi))
    \end{align*}
\end{theorem}
\begin{proof}
	One writes the functions in the series expansion in spherical monogenics. Subsequently one changes order of integration and summation, which is possible by Fubini's theorem and uses the orthonormality property of the $V_\alpha$ and $W_\alpha$.
\end{proof}
\begin{definition}
	The family of functions
	\begin{align*}
		\{\Psi_\rho(\xi) := \sum_{k=0}^\infty \sum_{|\alpha|=k} Exp\left(-\left(k(k+m)-\left({m+1}\atop 2\right)\right)\frac{t}{2}\right) (V_\alpha(\xi)+\alpha)W_\alpha(\xi))\},
	\end{align*}
	defines the diffusive wavelets corresponding to the modified diffusive operator $\Delta_\xi+\Gamma_\xi-\left({m+1}\atop 2\right){\bf 1}$.
	
	The corresponding wavelet transformation is given by
	\begin{align*}
		\mathcal W f(\varrho,g)  := \langle f(\cdot), \Psi_\varrho(g^{-1}(\cdot))\rangle_{L^2(S^m,\C_m)}
	\end{align*}
\end{definition}

\begin{theorem}
	The wavelet transform is invertible on its range:
	\begin{align*}
		f(\xi) &= \int_{t\to 0}^\infty \mathcal W f(\varrho,g) \ast \Psi_\varrho(g(\xi))\dx\varrho\; \forall f\in L^2(S^m,\C_m).
	\end{align*}
\end{theorem}
\begin{proof}
\begin{align*}
	\int_{t\to 0}^\infty \mathcal W f(\varrho,g) \ast \Psi_\varrho(g(\xi)) \dx\varrho &=   \int_{t\to 0}^\infty  \int_{SO(m)}\left(\int_{S^m} \overline{f(\zeta)} \Psi_\varrho(g^{-1}(\zeta))\dx\zeta \right) \Psi_\varrho(g^{-1}(\xi))\dx g \;\dx\varrho
\end{align*}
By construction we are dealing with an diffusive approximate identity, hence the change of order of integration is valid.
\begin{align*}
 &=\int_{t\to 0}^\infty  \int_{S^m} \overline{f(\zeta)} \left(\int_{SO(m)} \Psi_\varrho(g^{-1}(\zeta) \Psi_\varrho(g^{-1}(\xi))\right)\dx g )\dx\zeta\;\dx\varrho\\
&=\int_{S^m}   \overline{f(\zeta)} \left(\int_{t\to 0}^\infty  (\Psi_\varrho\ast \Psi_\varrho)(\zeta,\xi)\right)\;\dx\varrho\;\dx\zeta\\
&= \lim_{t\to 0}f \ast P_t(\xi) =f(\xi)
\end{align*}
\end{proof}

\subsection{Modification of the operator $L_\ast(\Omega)$}

It is an easy change to look at the operator $\Delta-\Gamma$ instead of $L_\ast(\Omega)$, we just replace the eigenvalues in the series expansion of the fundamental solution by the eigenvalues of $\Delta-\Gamma$, which is obviously $k(k+m)$, since these operators differ from each other only by a multiple of the identity operator
\begin{align*}
	\sum_{k=0}^\infty \sum_{|\alpha|=k} Exp\left(-k(k+m)t\right) (V_\alpha(\xi)+W_\alpha(\xi))
\end{align*}
The corresponding wavelets are now of the form
\begin{align*}
		\left\{\Psi_\rho(\xi) := \sum_{k=0}^\infty \sum_{|\alpha|=k} Exp\left(-k(k+m)\frac{t}{2}\right) (V_\alpha(\xi)+\alpha)W_\alpha(\xi))\right\},
\end{align*}

\subsection{Eigenfunction of $\Delta_{\spin}$ and the heat kernel on $\spin(m)$}\label{sec_eigenfct_on_Spin}

Let us now take a look at the case of the Spin group. Eigenfunctions of $\Delta_{\spin}$ on $\spin(m)$ can be given as matrix coefficients of eigenvectors of $\pi_\ast(\Omega)$, for any irreducible representation $\pi$.

All irreducible representations are of the form $H$ or $L$ in the subspace of simplicial harmonics or monogenics, respectively. For the moment we denote the eigenfunction with respect to the weight ${(l_1,...,l_{[\frac{m}{2}]})}$ by $v_{(l_1,...,l_{[\frac{m}{2}]})}$. Consequently, all functions of the form
\begin{align}
	h(s) &= \int_{\C_m} \overline{H(s)v_{(l_1,...,l_{[\frac{m}{2}]})}(a)}v_{(s_1,...,s_{[\frac{m}{2}]})}(a)\dx a,\qquad(l_1,...,l_{[\frac{m}{2}]}),(s_1,...,s_{[\frac{m}{2}]})\in (2\Z)^{[\frac{m}{2}]} \label{probl_int}\\
	l(s) &= \int_{\C_m} \overline{L(s)v_{(l_1,...,l_{[\frac{m}{2}]})}(a)}v_{(s_1,...,s_{[\frac{m}{2}]})}(a)\dx a,\qquad(l_1,...,l_{[\frac{m}{2}]}),(s_1,...,s_{[\frac{m}{2}]})\in ((2\Z+1)^{[\frac{m}{2}]}\label{probl_int_I}
\end{align}
represents harmonics and harmonic functions are linear combinations of them.

We can also chose the following way:

Since we already know the eigenfunctions of $H_\ast(\Omega)$ and $L_\ast(\Omega)$ if we can express $\Delta_{\spin}$ in terms of $H_\ast(\Omega)$ and $L_\ast(\Omega)$, then we easily obtain the eigenfunctions of $\Delta_{\spin}$.

This can be easily done for the Dirac operator on $\spin(m)$, which we denote by $\partial_s$. From Lemma $\ref{Lemma2}$ we know that a function on $\spin(m)$ can be regarded as a pair of functions $\alpha(g)$ and $\beta(g)$ on $g\in SO(m)$ or as a pair of a simplicial variable, respectively. Consequently, for a function $f$ on $\spin(m)$ we have
\begin{align*}
	f(s)&= H(s)\alpha(a(u_1,...,u_m)) + L(s)\beta(a(u_1,...,u_m)),
\end{align*}
where the simplicial variable $a(u_1,...,u_m)$ is fixed, in order to have the dependance on $s$ only.
For the action of the Dirac operator $\partial_s=\sum_{\C_{m,2}} e_{ij}(H_\ast(e_{ij})+L_\ast(e_{ij}))$ on $f$, by \eqref{H(e_ij)} and \eqref{L(e_ij)}  we have
\begin{align*}
	\partial_s f(s) &= \sum_{i<j} e_{ij}(H_\ast(e_{ij}) \alpha +L_\ast(e_{ij})\beta )\\
	&=\Gamma H(s)\alpha + (\Gamma -\left(m\atop{2}\right))L(s)\beta)
\end{align*}
Hence, we can  immediately deduce the eigensystem of $\partial_s$. The same construction we would like to have for $\Delta_s$. Therefore we look at the action of $\Delta_s$ on $H(s)\alpha$ and $L(s)\beta$ separately:
\begin{align}\label{Delat_{spin}}
	\Delta_{\spin} H(s)\alpha(a(u_1,...,u_m))&=(\sum_{j=1}^m\Delta_{u_j}+\sum_{k<l}\Delta_{u_k u_l}) H(s)\alpha(a(u_1,...,u_m))\\
	\Delta_{\spin} L(s)\beta(a(u_1,...,u_m))& =(\sum_{j=1}^m\Delta_{u_j}+\sum_{k<l}\Delta_{u_k u_l}+\sum_{j=1}^m\Gamma_{u_j}-\left(m\atop{2}\right)) L(s)\beta(a(u_1,...,u_m)).
\end{align}
Since for the Laplacian in the components $u_m$ we have
\begin{align*}
	\Delta_u=\sum_{i<j}L^2_{u,e_{ij}}=\Gamma_u(m-2-\Gamma_u),\mbox{ with }\Gamma_u=u\wedge \partial _u,
\end{align*} the only critical point is the study of the part of the mixed Laplacian
\begin{align*}
	 \Delta_{uv} =\sum_{i<j}L_{u,e_{ij}}L_{v,e_{ij}}.
\end{align*}
To this end we can express the action of $\Delta_{uv}$ on monogenics in terms of $u,v, \partial_u$ and $\partial_v$, as we did for $\Delta_u$.
A rather technical calculation, which can be found in \cite{preprint_BES}
gives
\begin{align*}
    \Delta_{uv}f(u,v) &= -<v,\dot{\partial}_u><u,\partial_v>\dot{f}(u,v).
\end{align*}

where the dot means, that the derivative $\partial_u$ is applied directly to $f(u,v)$, but not to $\langle u,\partial_v\rangle$ (Hestenes overdot notation).

Consequently, we have
\begin{align*}
	&\Delta_{\spin} H(s)\alpha (a(u_1,...,u_m))\\
&= \left(\sum_{j=1}^m \Gamma_{u_j}(m-2-\Gamma_{u_j})-\sum_{k<l}\langle u_{k},\dot{\partial_{u_l}}\rangle \langle u_l,\partial_{u_k}\rangle  \right)H(s)\dot{\alpha}(a(u_1,...,u_m))\\
\\
&\Delta_{\spin} L(s)\beta (a(u_1,...,u_m))\\
	&=\left(\sum_{j=1}^m \Gamma_{u_j}(m-2-\Gamma_{u_j})-\sum_{k<l}\langle u_{k},\dot{\partial_{u_l}}\rangle \langle u_l,\partial_{u_k}\rangle  +\sum_{j=1}^m\Gamma_{u_j}-\left(m\atop{2}\right) \right) L(s)\dot{\beta}(a(u_1,...,u_m))
\end{align*}	
A closer look to the operator $\langle u_{k},\dot{\partial_{u_l}}\rangle$ shows that
\begin{align*}
    \langle u,\partial_{v}\rangle &=\sum_{i=1}^m u_i\partial_{v_{i}},
\end{align*}
which can be viewed as a mixed Euler operator c.f.\cite{preprint_BES}. In fact, from the characteristic system of simplicial monogenics \eqref{simplicial_system} we know that simplicial functions vanish under the mixed Euler operator.

We discussed already simplicial monogenics in Section \ref{sec_simpl_fct}. Let $k_1,...,k_m$ ($l_1,...,l_m$)  denote the degree of homogeneity of $\alpha$ ( or $\beta$) in the variable $u_1,...,u_m$, respectively. Therefore, $\Gamma_{u_i}(H(s)\alpha+L(s)\beta)=(k_i\alpha+l_i\beta)$. Hence for functions $f(s)=H(s)\alpha+L(s)\beta$ on $\spin(m)$ we have
\begin{align*}
	&\Delta_{\spin} H(s)\alpha( a(u_1,...,u_m))\\
&= \left(\sum_{j=1}^m \Gamma_{u_j}(m-2-\Gamma_{u_j})\right)H(s)\dot{\alpha}(a(u_1,...,u_m))\\
&= \left(\sum_{j=1}^m k_j(m-2-k_j)\right) H(s)\alpha(a(u_1,...,u_m))
\\
\end{align*}
and
\begin{align*}
&\Delta_{\spin} L(s)\beta (a(u_1,...,u_m))\\
	&=\left(\sum_{j=1}^m \Gamma_{u_j}(m-2-\Gamma_{u_j})+\sum_{j=1}^m\Gamma_{u_j}-\left(m\atop{2}\right) \right) L(s)\dot{\beta}(a(u_1,...,u_m))\\
&= \left(\sum_{j=1}^m l_j(m-2-l_j) \right)L(s)\beta(a(u_1,...,u_m)).
\end{align*}	

Such that according to our construction of wavelets we obtain the following.

\begin{theorem}
Clifford-valued diffusive wavelets on $\spin(m)$ assume the form
\begin{align*}
    \psi_\rho(s) &= \sum_{k=1}^\infty \sum_{\mathfrak m \in \Z^k}\exp\left(-\left(\sum_{j=1}^m k_j(m-2-k_j) \right)\frac{t}{2}\right) H(s)\mathcal K_\mathfrak m \\
    & \qquad\qquad  + \exp\left(-\left(\sum_{j=1}^m k_j(m-2-k_j)\right)\frac{t}{2}\right) L(s)\mathcal L_{\mathfrak m}
\end{align*}
where $\mathcal K_\mathfrak m$ and $\mathcal L_{\mathfrak m}$ form a complete system of simplicial functions. (see \eqref{simpl_harm} and \eqref{simpl_mon}).
\end{theorem}

\bibliography{Literatur}

\begin{thebibliography}{10}

\bibitem{ABLSS}
W.W. Adams, C.~A. Berenstein, P.~Loustaunau, I.~Sabadini, and D.C. Struppa.
\newblock Regular functions of several quaternionic variables.
\newblock {\em J. Geom. Anal}, 9:1--15, 1999.

\bibitem{Ant/VanN}
J.-P. Antoine and P.~Vandergheynst.
\newblock Wavelets on the n-sphere and related manifolds.
\newblock {\em Journal of {M}athematical {P}hysics}, 39(8):3987--4008, 1998.

\bibitem{Ant/Van}
J.-P. Antoine and P.~Vandergheynst.
\newblock Wavelets on the {$2$}-sphere: a group-theoretical approach.
\newblock {\em Appl. Comput. Harmon. Anal.}, 7(3):262--291, 1999.

\bibitem{Ant/Van2}
J.-P. Antoine and P.~Vandergheynst.
\newblock Wavelets on the two-sphere and other conic sections.
\newblock {\em J. Fourier Anal. Appl.}, 13(4):369--386, 2007.

\bibitem{BE-MMAS}
S.~Bernstein and S.~Ebert.
\newblock {Wavelets on $S^3$ and $SO(3)$}.
\newblock {\em Mathematical Methods in the Applied Sciences}, 2009.

\bibitem{MMAS_I}
S.~Bernstein and S.~Ebert.
\newblock {Wavelets on $S^3$ and $SO(3)$ - their construction, relation to each
  other and Radon transform of Wavelets on SO(3)}.
\newblock {\em MMAS}, 2010.

\bibitem{preprint_BES}
S.~Bernstein, S.~Ebert, and F.~Sommen.
\newblock Diffusive wavelets and clifford analysis.
\newblock {\em Preprint TU Bergakademie Freiberg, Fakult{\"a}t f{\"u}r
  Mathematik und Informatik}, 2011-05, 2011.

\bibitem{bernstein/schaeben}
S.~Bernstein and H.~Schaeben.
\newblock A one-dimensional radon transform on {$SO(3)$} and its application to
  texture goniometry.
\newblock {\em Mathematical Methods in the Applied Sciences}, 28:1269--1289,
  July 2005.

\bibitem{D_Bump}
D.~Bump.
\newblock {\em Lie Groups}.
\newblock Springer-Verlag, 2004.

\bibitem{cartan}
E.~Cartan.
\newblock {\em The Theory of Spinors}.
\newblock Dover, 1981.

\bibitem{Co}
D.~Constales.
\newblock {\em The relative position of $L_2$-domains in complex and Clifford
  analysis}.
\newblock PhD thesis, State Uni. Ghent, 1992.

\bibitem{DSS}
R.~Delanghe, F.~Sommen, and V.~Sou{\v{c}}ek.
\newblock {\em Clifford algebra and spinor-valued functions, A function theory
  for the Dirac operator.}, volume~53 of {\em Mathematics and its
  Applications}.
\newblock Kluwer Academic Publishers Group, Dordrecht, 1992.

\bibitem{DoHeSoA}
C.~Doran, D.~Hestenes, F.~Sommen, and N.~Van Acker.
\newblock Lie groups as spin groups.
\newblock {\em J. Math. Phys.}, 34(8):3642--3669, 1993.

\bibitem{Ebert}
S.~Ebert.
\newblock {\em Wavelets on Lie groups and homogenuous spaces}.
\newblock PhD thesis, TU Bergakademie Freiberg, Department of Mathematics and
  Informatics, 2011.

\bibitem{ebert/wirth}
S.~Ebert and J.~Wirth.
\newblock Diffusive wavelets on compact lie groups and homogeneous spaces.
\newblock {\em Proc. R. Soc. Edinb., Sect. A, Math.}, 141(3):497--520, 2011.

\bibitem{Intro_comp_lie_groups}
H.~D. Fegan.
\newblock {\em Introduction to compact {L}ie groups}, volume~13 of {\em Series
  in Pure Mathematics}.
\newblock World Scientific Publishing Co. Inc., River Edge, NJ, 1991.

\bibitem{Freeden/Windheuser}
W.~Freeden and U.~Windheuser.
\newblock Combined spherical harmonic and wavelet expansin- a future concept in
  earth's gravitational determination.
\newblock {\em Laboratory of Technomethematics, Geomathematics Group,
  University of Kaiserslautern}, 1996.

\bibitem{Geller/Mayeli}
D.~Geller and A.~Mayeli.
\newblock Continous wavelets and frames on stratified lie groups i.
\newblock {\em J. of Fourier Analysis and Appl.}, 12:543--579, 2006.

\bibitem{GM}
J.~E. Gilbert and M.~A.~M. Murray.
\newblock {\em Clifford algebras and {D}irac operators in harmonic analysis},
  volume~26 of {\em Cambridge Studies in Advanced Mathematics}.
\newblock Cambridge University Press, Cambridge, 1991.

\bibitem{GrossRichards}
K.~I. Gross and D.~St.~P. Richards.
\newblock Special functions of matrix argument. {I}. {A}lgebraic induction,
  zonal polynomials, and hypergeometric functions.
\newblock {\em Trans. Amer. Math. Soc.}, 301(2):781--811, 1987.

\bibitem{JHla}
J.~Hladik.
\newblock {\em Spinors in Physics}.
\newblock Springer, New York, 1999.

\bibitem{holschneider}
M.~Holschneider.
\newblock Continuous wavelet transforms on the sphere.
\newblock {\em J. Math. Phys.}, 37(8):4156--4165, 1996.

\bibitem{VanSomCon}
P.~Van Lancker, F~Sommen, and D~Constales.
\newblock Models of irreducible representations of {$Spin(m)$}.
\newblock {\em Advances in applied Clifford algebras}, 11:271--289, 2001.

\bibitem{Pa}
V.P. Palamodov.
\newblock Holomorphic synthesis of monogenic functions of several quaternionic
  variables.
\newblock {\em J. D'Analyse Mathématique}, 78 (1):177--204, 1999.

\bibitem{Pe}
D.~Pertici.
\newblock {\em Teoria delle funzioni di piu variabili quaternionici}.
\newblock PhD thesis, Univ. D. Stud. d. Firenze, 1988.

\bibitem{SO1}
F.~Sommen.
\newblock Functions on the spingroup.
\newblock {\em Adv. Appl. Clifford Algebras}, 6(1):37--48, 1996.

\bibitem{BHPS}
K.~G. van~den Boogaart, R.~Hilscher, J.~Prestin, and H.~Schaeben.
\newblock Application of the radial basis function method to texture analysis.
\newblock 2004.

\bibitem{So_La_Co}
P.~Van~Lancker, F.~Sommen, and D.~Constales.
\newblock Models for irreducible representations of {${\rm Spin}(m)$}.
\newblock {\em Adv. Appl. Clifford Algebras}, 11(S1):271--289, 2001.

\bibitem{Vilenkin_I}
N.~Ja. Vilenkin and A.~U. Klimyk.
\newblock {\em Representation of {L}ie groups and special functions. {V}ol. 1},
  volume~72 of {\em Mathematics and its Applications (Soviet Series)}.
\newblock Kluwer Academic Publishers Group, Dordrecht, 1991.
\newblock Simplest Lie groups, special functions and integral transforms,
  Translated from the Russian by V. A. Groza and A. A. Groza.

\bibitem{Vilenkin_III}
N.~Ja. Vilenkin and A.~U. Klimyk.
\newblock {\em Representation of {L}ie groups and special functions. {V}ol. 3},
  volume~75 of {\em Mathematics and its Applications (Soviet Series)}.
\newblock Kluwer Academic Publishers Group, Dordrecht, 1992.
\newblock Classical and quantum groups and special functions, Translated from
  the Russian by V. A. Groza and A. A. Groza.

\bibitem{Vilenkin}
N.~Ja. Vilenkin and A.~U. Klimyk.
\newblock {\em Representation of {L}ie groups and special functions. {V}ol. 2},
  volume~74 of {\em Mathematics and its Applications (Soviet Series)}.
\newblock Kluwer Academic Publishers Group, Dordrecht, 1993.
\newblock Class I representations, special functions, and integral transforms,
  Translated from the Russian by V. A. Groza and A. A. Groza.

\end{thebibliography}
\bibliographystyle{plain}
\end{document}